\documentclass[12pt]{amsart}

\usepackage{amssymb}

\theoremstyle{plain}
  \newtheorem{theorem}{Theorem}[section]
  \newtheorem{proposition}[theorem]{Proposition}
  \newtheorem{lemma}[theorem]{Lemma}
  \newtheorem{corollary}[theorem]{Corollary}
\theoremstyle{definition}
  \newtheorem{definition}[theorem]{Definition}
  \newtheorem{example}[theorem]{Example}
  
  \newtheorem{remark}[theorem]{Remark}
  
\newtheorem{conjecture}[theorem]{Conjecture}
\newtheorem{problem}[theorem]{Problem}

\newcommand{\res}{\mathrm{res}}
\newcommand{\sd}{\mathrm{sd}}
\newcommand{\cone}{\mathrm{cone}}
\newcommand{\des}{\mathrm{des}}

\newcommand{\KER}{\mathrm{Ker}}
\newcommand{\Tor}{\mathrm{Tor}}

\begin{document}

\title[Barycentric subdivisions of shellable complexes]{The Lefschetz property for barycentric subdivisions of shellable complexes}

\author{Martina Kubitzke}
\address{Fachbereich Mathematik und Informatik\\
      Philipps-Universit\"at Marburg\\
      35032 Marburg, Germany and
\thanks{Department of Mathematics\\
    Cornell University\\
    Ithaca, NY 14853, USA}}
\email{kubitzke@mathematik.uni-marburg.de}

\author{Eran Nevo}
\address{Department of Mathematics\\
    Cornell University\\
    Ithaca, NY 14853, USA}
\email{eranevo@math.cornell.edu}

\thanks{Martina Kubitzke was supported by DAAD}



\begin{abstract}
We show that an 'almost strong Lefschetz' property holds for the barycentric subdivision of a shellable complex. {From} this we conclude that for the barycentric subdivision of a Cohen-Macaulay complex, the $h$-vector is unimodal, peaks in its middle degree (one of them if the dimension of the complex is even), and that its $g$-vector is an $M$-sequence. In particular, the (combinatorial) $g$-conjecture is verified for barycentric subdivisions of homology spheres.
In addition, using the above algebraic result, we derive new inequalities on a refinement of the Eulerian statistics on permutations, where permutations are grouped by the number of descents and the image of $1$.
\end{abstract}

\maketitle

\section{Introduction}

The starting point for this paper is Brenti and Welker's study of $f$-vectors of barycentric subdivisions of simplicial complexes \cite{Brenti-Welker}. They showed that for a Cohen-Macaulay complex, the $h$-vector of its barycentric subdivision is unimodal (\cite[Corollary 3.5]{Brenti-Welker}). This raises the following natural questions about this $h$-vector: Where is its peak? Is the vector of its successive differences up to the middle degree ('$g$-vector') an $M$-sequence?

We answer these questions by finding an 'almost strong Lefschetz' element in case the original complex is shellable. Let us make this precise (for unexplained terminology see Section \ref{sec:Proof}):
let $\Delta$ be a $(d-1)$-dimensional Cohen-Macaulay simplicial complex over a field $k$, on vertex set $[n]:=\{1,\ldots,n\}$ and $\Theta=\{\theta_1,...,\theta_d\}$ a maximal linear system of parameters for its face ring $k[\Delta]$. We call a degree one element in the polynomial ring $\omega\in A=k[x_1,\ldots,x_n]$ an \emph{$s$-Lefschetz element for the $A$-module $k[\Delta]/\Theta$} if multiplication $$\omega^{s-2i}:\hspace{5pt}(k[\Delta]/\Theta)_i\;\longrightarrow\;(k[\Delta]/\Theta)_{s-i}$$
is an injection for $0\leq i\leq \lfloor\frac{s-1}{2}\rfloor$. A $(\dim\Delta)$-Lefschetz element is called an \emph{almost strong Lefschetz element for $k[\Delta]/\Theta$}.
(Recall that $k[\Delta]/\Theta$ is called strong Lefschetz if it has a $(\dim\Delta +1)$-Lefschetz element.)
Let $G_s(\Delta)$ be the set consisting of all pairs $(\Theta,\omega)$ such that $\Theta$ is a maximal linear system of parameters for $k[\Delta]$ and $\omega$ is an $s$-Lefschetz element for $k[\Delta]/\Theta$. It can be shown that $G_s(\Delta)$ is a Zariski open set (e.g. imitate the proof in \cite[Proposition 3.6]{Swartz}).
If $G_s(\Delta)\neq \emptyset$ we say that $\Delta$ is \emph{$s$-Lefschetz over $k$}, and that $\Delta$ is \emph{almost strong Lefschetz over $k$} if $G_{\dim\Delta}(\Delta)\neq \emptyset$.

\begin{theorem} \label{mainresult}
Let $\Delta$ be a shellable $(d-1)$-dimensional simplicial complex and let $k$ be an infinite field. Then the barycentric subdivision of $\Delta$ is almost strong Lefschetz over $k$.
\end{theorem}

This theorem has some immediate $f$-vector consequences; in particular it verifies the $g$-conjecture for barycentric subdivisions of homology spheres, and beyond.
 One of the main problems in algebraic combinatorics is the $g$-conjecture, first raised as question by McMullen for simplicial spheres \cite{McMullen-g-conj}. Here we state the part of the conjecture which is still open.
\begin{conjecture}($g$-conjecture)\label{conj-g}
Let $L$ be a simplicial sphere, then its $g$-vector is an $M$-sequence.
\end{conjecture}
It is conjectured to hold in greater generality, for all homology spheres and even for all doubly Cohen-Macaulay complexes, as was suggested by Bj\"{o}rner and Swartz \cite{Swartz}.
We verify these conjectures in a special case, as it was already conjectured in \cite{BW-conjectures}:
\begin{corollary}\label{cor:CM,M-seq}
Let $\Delta$ be a Cohen-Macaulay simplicial complex (over some field). Then
the $g$-vector of the barycentric subdivision of $\Delta$ is an $M$-sequence.
In particular, the $g$-conjecture holds for barycentric subdivisions of simplicial spheres, of homology spheres, and of doubly Cohen-Macaulay complexes.
\end{corollary}
Note that the non-negativity of the $g$-vector of barycentric subdivisions of homology spheres already follows from Karu's result on the non-negativity of the cd-index for order complexes of Gorenstein$^*$ posets \cite{Karu-cd-index}.
In Section \ref{sec:Proof} we provide some preliminaries and prove our main result Theorem \ref{mainresult}.
In Section \ref{CombCons} we derive some $f$-vector corollaries from Theorem \ref{mainresult}, as well as extending this theorem to shellable polytopal complexes.
In Section \ref{permutations}
we prove new inequalities for the refined Eulerian statistics on permutations, introduced by Brenti and Welker. The proofs are based on Theorem \ref{mainresult}. As a corollary, the location of the peak of the $h$-vector of the barycentric subdivision of a Cohen-Macaulay complex is determined.

\section{Proof of Theorem \ref{mainresult}}\label{sec:Proof}
Let $\Delta$ be a finite (abstract, non-empty) simplicial complex on a vertex set $\Delta_0=[n]=\{1,2,...,n\}$, i.e. $\Delta \subseteq 2^{[n]}$ and if
$S\subseteq T\in \Delta$ then $S\in \Delta$ (and $\emptyset \in \Delta$), and let $\Delta$ be of dimension $d-1$, i.e. $\max_{S\in \Delta}\#S=d$. The $f$-vector of $\Delta$ is $f^{\Delta}=(f_{-1},f_0,...,f_{d-1})$ where $f_{i-1}=\#\{S\in \Delta: \#S=i\}$. its $h$-vector, which carries the same combinatorial information, is $h^{\Delta}=(h_0,...,h_d)$, where $\sum_{0\leq i\leq d}h_i x^{d-i}= \sum_{0\leq i\leq d}f_{i-1}(x-1)^{d-i}$.

Let $k$ be an infinite field and $A=k[x_1,...,x_n]=A_0\oplus A_1\oplus...$ the polynomial ring graded by degree. The face ring (Stanley-Reisner ring) of $\Delta$ over $k$ is $k[\Delta]=A/I_{\Delta}$ where $I_{\Delta}$ is the ideal $I_{\Delta}=(\prod_{1\leq i \leq n}x_i^{a_i}: \{i:a_i>0\}\notin \Delta)$. It inherits the grading from $A$.
If $k$ is infinite, a maximal homogeneous system of parameters for $k[\Delta]$ can be chosen from the 'linear' part $k[\Delta]_1$, called l.s.o.p. for short. If $k[\Delta]$ is Cohen-Macaulay (CM for short) then any maximal l.s.o.p has $d$ elements. We will denote such l.s.o.p. by $\Theta=\{\theta_1,...,\theta_d\}$. In this case $k[\Delta]$ is a free $k[\Theta]$-module, and $h_i^{\Delta}=\dim_k(k[\Delta]/\Theta)_i$. We say that $\Delta$ is CM (over $k$) if $k[\Delta]$ is a CM ring.

The \emph{barycentric subdivision} of a simplicial complex $\Delta$ is the simplicial complex $\sd(\Delta)$ on vertex set $\Delta\setminus\{\emptyset\}$ whose simplices are all the chains $F_0\subsetneq F_1\subsetneq\ldots\subsetneq F_r$ of elements $F_i\in\Delta\setminus\{\emptyset\}$ for $0\leq i\leq r$. The geometric realizations of $\Delta$ and $\sd(\Delta)$ are homeomorphic. Recall that Cohen-Macaulayness is a topological property. Hence, if $\Delta$ is CM then $\sd(\Delta)$ is CM as well.
In this case Baclawski and Garsia (\cite[Proposition 3.4.]{Baclawski-Garsia}) showed that $\Theta=\{\theta_1,\ldots,\theta_{\dim\Delta+1}\}$ where $\theta_i:=\sum_{F\in\Delta,\#F=i}x_{\{F\}}$ for $1\leq i\leq\dim\Delta+1$, is a l.s.o.p. for $k[\sd(\Delta)]$.
For further terminology and background we refer to Stanley's book \cite{StanleyGreenBook}.
Let us start with some auxiliary results.

We denote by $\cone(\Delta)$ the cone over $\Delta$, i.e. $\cone(\Delta)$ is the join of some
vertex $\{v\}$ with $\Delta$ where $v\notin \Delta$, $\cone(\Delta):=\{~F~|~F\in \Delta\}\cup\{~F\cup\{v\}~|~F\in \Delta\}$.
The following lemma deals with the effect of coning on the $s$-Lefschetz property.

\begin{lemma}\label{SL:cone}
Let $\Delta$ be a $(d-1)$-dimensional simplicial complex. If $\Delta$ is $s$-Lefschetz over $k$ then the same is true for $\cone(\Delta)$.
\end{lemma}

\begin{proof}
Let $\Theta$ be a l.s.o.p. for $k[\Delta]$ and let $v$ be the apex of $\cone(\Delta)$.
Then $\widetilde{\Theta}:=\Theta\cup \{x_v\}$ is a l.s.o.p. for $k[\cone(\Delta)]$ (to see this, use the isomorphism $k[\cone(\Delta)]\cong k[\Delta]\otimes_k k[x_v]$ of modules over $k[x_u:u\in \Delta_0\cup \{v\}]\cong k[x_u:u\in \Delta_0]\otimes_k k[x_v]$).
Furthermore, $k[\Delta]/\Theta\cong k[\cone(\Delta)]/\widetilde{\Theta}$ as $A$-modules, where $A=k[x_i: i\in \{v\}\cup \Delta_0]$ and $x_v\cdot k[\Delta]=0$.
Hence, for any pair $(\Theta,\omega)\in G_s(\Delta)$ we have $(\widetilde{\Theta},\omega)\in G_s(\cone(\Delta))$, and the assertion follows.
\end{proof}

Note that if $\Delta$ is almost strong Lefschetz over $k$ then $\cone(\Delta)$ is $(\dim\Delta)$-Lefschetz over $k$.

The following theorem is the main part of Stanley's proof of the necessity part of the $g$-theorem for simplicial polytopes \cite{St}.

\begin{theorem}[\cite{St}]\label{thm:HLforPolytopes}
Let $P$ be a simplicial $d$-polytope and let $\Delta$ be the boundary complex of $P$. Then $\Delta$ is $d$-Lefschetz over $\mathbb{R}$.
\end{theorem}

If $\Delta$ is a simplicial complex and $\{F_1,\ldots, F_m\}\subseteq\Delta$ is a collection of faces of $\Delta$ we denote by $\langle F_1,\ldots,F_{m}\rangle$ the simplicial complex whose faces are the subsets of the $F_i$'s, $1\leq i\leq m$.

For an arbitrary infinite field $k$ (of arbitrary characteristic!) the conclusion in Theorem \ref{thm:HLforPolytopes} holds for the following polytopes, which will suffice for concluding our main result Theorem \ref{mainresult}:

\begin{proposition}\label{prop:LefSdSimplex}
Let $P$ be a $d$-simplex and let $\Delta$ be its barycentric subdivision. Let $k$ be an infinite field. Then $\Delta$ is almost strong Lefschetz over $k$.
\end{proposition}

\begin{proof}
Note that the boundary complex $\partial \Delta$ is obtained from $\partial P$ by a sequence of stellar subdivisions - order the faces of $\partial P$ by decreasing dimension and perform a stellar subdivision at each of them according to this order to obtain $\partial \Delta$. In particular, $\partial \Delta$ is \emph{strongly edge decomposable}, introduced in \cite{Nevo-VK}, as the inverse stellar moves when going backwards in this sequence of complexes demonstrate.

It was shown by Murai \cite[Corollary 3.5]{Murai-EdgeDecomposable} that
strongly edge decomposable complexes have the strong Lefschetz property (see also \cite[Corollary 4.6.6]{Nevo-PhD}).
As $\Delta=\cone(\partial \Delta)$, we conclude that $\Delta$ is
$d$-Lefschetz over $k$ by Lemma \ref{SL:cone}.
\end{proof}

We would like to point out that the proof of Proposition \ref{prop:LefSdSimplex} is self-contained and does not require Theorem \ref{thm:HLforPolytopes}.
Shellability of simplicial complexes is a useful tool in combinatorics; here we give two equivalent definitions for shellability which we will use later.
\begin{definition}
A pure simplicial complex $\Delta$ is called \emph{shellable} if $\Delta$ is a simplex or if one of the following equivalent conditions is satisfied. There exists a linear ordering $F_1,\ldots,F_m$ of the facets of $\Delta$ such that
\begin{itemize}
\item [(a)] $\langle F_i\rangle\cap\langle F_1,\ldots,F_{i-1}\rangle$ is generated by a non-empty set of maximal proper faces of $\langle F_i\rangle$, for all $2\leq i\leq m$.
\item [(b)] the set $\{F~|~F\in\langle F_1,\ldots, F_i\rangle, F\notin\langle F_1,\ldots,F_{i-1}\rangle\}$ has a unique minimal element for all $2\leq i\leq m$. This element is called the \emph{restriction face} of $F_i$. We denote it by $\res(F_i)$.
\end{itemize}
A linear order of the facets satisfying the equivalent conditions $(a)$ and $(b)$ is called a \emph{shelling} of $\Delta$.
\end{definition}

We are now in position to prove Theorem \ref{mainresult}. In the sequel, we will loosely use the term 'generic elements' to mean that these elements are chosen from a Zariski non-empty open set, to be understood from the context.

\begin{proof} [Proof of Theorem \ref{mainresult}]
The proof is by double induction, on the number of facets $f_{\dim\Delta}^{\Delta}$ of $\Delta$ and on the dimension of $\Delta$.
Let $\dim\Delta\geq 0$ be arbitrary and $f_{\dim\Delta}^{\Delta}=1$, i.e. $\Delta$ is a $(d-1)$-simplex, and by Proposition \ref{prop:LefSdSimplex} we are done.
Let $\dim\Delta=0$, i.e. $\Delta$ as well as $\sd(\Delta)$ consist of vertices only. Since $h_0^{\sd(\Delta)}=h_{1-1-0}^{\sd(\Delta)}$ there is nothing to show.
This provides the base of the induction.

For the induction step let $\dim\Delta\geq 1$. Let $n=f_0^{\sd(\Delta)}$ and let $A=k[x_1,\ldots,x_n]$ be the polynomial ring in $n$ variables.
Let $F_1,\ldots,F_m$ be a shelling of $\Delta$ with $m\geq 2$ and let $\widetilde{\Delta}:=\langle F_1,\ldots,F_{m-1}\rangle$. Then $\sigma:=\widetilde{\Delta}\cap \langle F_m \rangle$ is a pure $(d-2)$-dimensional subcomplex of $\partial F_m$. The barycentric subdivision $\sd(\Delta)$ of $\Delta$ is given by $\sd(\Delta)=\sd(\widetilde{\Delta})\cup\sd(\langle F_m\rangle)$ and $\sd(\sigma)=\sd(\widetilde{\Delta})\cap\sd(\langle F_m\rangle)$.

We get the following Mayer-Vietoris exact sequence of $A$-modules:
\begin{equation}\label{eq:MV-seq}
0\rightarrow k[\sd(\Delta)]\rightarrow k[\sd(\widetilde{\Delta})]\oplus k[\sd(\langle F_m \rangle)]\rightarrow k[\sd(\sigma)]\rightarrow 0.
\end{equation}
Here the injection on the left-hand side is given by $\alpha\mapsto (\tilde{\alpha},-\tilde{\alpha})$ and the surjection on the right-hand side by $(\beta,\gamma)\mapsto \tilde{\beta}+\tilde{\gamma}$, where $\tilde{a}$ denotes the obvious projection of $a$ on the appropriate quotient module.
(For a subcomplex $\Gamma$ of $\Delta$ and $v\in \Delta_0\setminus \Gamma_0$ it holds that $x_v\cdot k[\Gamma]=0$.)

Let $\Theta=\{\theta_1,\ldots,\theta_d\}$ be a (maximal) l.s.o.p. for $k[\sd(\Delta)]$, $k[\sd(\widetilde{\Delta})]$ and $k[\sd(F_m)]$, and such that $\{\theta_1,\ldots,\theta_{d-1}\}$ is a l.s.o.p. for $k[\sigma]$. This is possible, as the intersection of finitely many non-empty Zariski open sets is non-empty (for $k[\sigma]$, its set of maximal l.s.o.p.'s times $k^n$ (for $\theta_d$) is Zariski open in $k^{dn}$).
Dividing out by $\Theta$ in the short exact sequence (\ref{eq:MV-seq}),
which is equivalent to tensoring with $-\otimes_A A/\Theta$, yields the following $\Tor$-long exact sequence:
\begin{eqnarray*}
&\ldots&\rightarrow \Tor_1(k[\sd(\Delta)],A/\Theta)\rightarrow \Tor_1(k[\sd(\widetilde{\Delta})]\oplus k[\sd(\langle F_m\rangle)],A/\Theta)\\
&\rightarrow& \Tor_1(k[\sd(\sigma)],A/\Theta)\stackrel{\delta}{\rightarrow} \Tor_0(k[\sd(\Delta)],A/\Theta)\\
&\rightarrow& \Tor_0(k[\sd(\widetilde{\Delta})]\oplus k[\sd(\langle F_m\rangle)],A/\Theta)\rightarrow \Tor_0(k[\sd(\sigma)],A/\Theta)\rightarrow 0,
\end{eqnarray*}
where $\delta:\hspace{5pt}\Tor_1(k[\sd(\sigma)],A/\Theta)\rightarrow \Tor_0(k[\sd(\Delta)],A/\Theta)$ is the connecting homomorphism.
Below we write $k(\sd(\Delta))$ for $k[\sd(\Delta)]/\Theta$, and similarly $k(\sd(\widetilde{\Delta}))$, $k(\sd(\sigma))$ and $k(\sd(\langle F_m\rangle))$ for $k[\sd(\widetilde{\Delta})]/\Theta$, $k[\sd(\sigma)]/\Theta$ and $k[\sd(\langle F_m\rangle)]/\Theta$ resp.

Using that for $R$-modules $M$, $N$ and $Q$ it holds that $\Tor_0(M,N)\cong M\otimes_R N$, $(M\oplus N)\otimes_R Q\cong(M\otimes_R Q)\oplus (N\otimes_R Q)$ and that $M/IM\cong M\otimes_R R/I$ for an ideal $I\triangleleft R$,
we get the following exact sequence of $A$-modules:
\begin{eqnarray*}
\Tor_1(k[\sd(\sigma)],A/\Theta)\stackrel{\delta}{\rightarrow} k(\sd(\Delta))\rightarrow k(\sd(\widetilde{\Delta}))\oplus k(\sd(\langle F_m\rangle))\rightarrow k(\sd(\sigma))\rightarrow 0.
\end{eqnarray*}
Note that all the maps in this sequence are grading preserving, where $\Tor_1(k[\sd(\sigma)],A/\Theta)$ inherits the grading from (a projective grading preserving resolution of) the sequence (\ref{eq:MV-seq}).
{From} this we deduce the following commutative diagram:
$$\begin{array}{ccccccccc}
   \Tor_1(k(\sd(\sigma)))_i & \stackrel{\delta}{\to} & k(\sd(\Delta))_i &\to& k(\sd(\widetilde{\Delta}))_i\oplus k(\sd(\langle F_m\rangle))_i\\
    & & & & & & &\\
   & & \downarrow \omega^{d-2i-1}& &\ \ \downarrow (\omega^{d-2i-1},\omega^{d-2i-1})& & \\
     & & & & & & &\\
      &\empty &  k(\sd(\Delta))_{d-1-i} &\to& k(\sd(\widetilde{\Delta}))_{d-1-i}\oplus k(\sd(\langle F_m\rangle))_{d-1-i}
   \end{array}$$   where $\omega$ is a degree one element in $A$.
   Since $F_m$ is a $(d-1)$-simplex we know from the base of the induction that multiplication  $$\omega^{d-2i-1}:\hspace{5pt} k(\sd(\langle F_m\rangle))_i\rightarrow k(\sd(\langle F_m\rangle))_{d-1-i}$$
is an injection for a generic choice of $\omega$ in $A_1$.
(Note that if $G$ is a Zariski open set in $k[x_v:v\in (F_m)_0]_1$ then $G\times k[x_v:v\in\Delta_0\setminus (F_m)_0]_1$ is Zariski open in $A_1$.)

By construction, $\widetilde{\Delta}$ is shellable and therefore by the induction hypothesis the multiplication
$$\omega^{d-2i-1}:\hspace{5pt} k(\sd(\widetilde{\Delta}))_i\rightarrow k(\sd(\widetilde{\Delta}))_{d-1-i}$$
is an injection for generic $\omega$. Since the intersection of two non-empty Zariski open sets is non-empty, multiplication
   $$(\omega^{d-2i-1},\omega^{d-2i-1}):\hspace{1pt} k(\sd(\widetilde{\Delta}))_i\oplus k(\sd(\langle F_m\rangle))_i\rightarrow k(\sd(\widetilde{\Delta}))_{d-1-i}\oplus k(\sd(\langle F_m\rangle))_{d-1-i}$$
   is an injection for a generic $\omega\in A_1$.

   Our aim is to show that $\Tor_1(k[\sd(\sigma)],A/\Theta)_i=0$ for $0\leq i\leq \lfloor \frac{d-2}{2}\rfloor$.
   As soon as this is shown,
   the above commutative diagram implies that multiplication
   $$\omega^{d-2i-1}:\hspace{5pt} k(\sd(\Delta))_i\rightarrow k(\sd(\Delta))_{d-1-i}$$
   is injective for $0\leq i\leq \lfloor\frac{d-2}{2}\rfloor$ and $\omega$ as above.

   For the computation of $\Tor_1(k[\sd(\sigma)],A/\Theta)$ we consider the following exact sequence of $A$-modules:
   $$0\rightarrow\Theta A\rightarrow A\rightarrow A/\Theta\rightarrow 0.$$
   Since $\Tor_0(M,N)\cong M\otimes_R N$ and $\Tor_1(R,M)=0$ for $R$-modules $M$ and $N$, we get the following $\Tor$-long exact sequence
   $$0\rightarrow \Tor_1(A/\Theta,k[\sd(\sigma)])\rightarrow \Theta A\otimes_A k[\sd(\sigma)]\rightarrow  k[\sd(\sigma)]\rightarrow A/\Theta\otimes_A k[\sd(\sigma)]\rightarrow 0.$$
  {From} the exactness of this sequence we deduce $\Tor_1(A/\Theta,k[\sd(\sigma)])=\KER(\Theta A\otimes_A k[\sd(\sigma)]\rightarrow k[\sd(\sigma)])$.
  Since we have $\Tor_1(k[\sd(\sigma)],A/\Theta)\cong \Tor_1(A/\Theta,k[\sd(\sigma)])$, and by the fact that the isomorphism is grading preserving, we finally get that
   $\Tor_1(k[\sd(\sigma)],A/\Theta)\cong\KER(\Theta A\otimes_A k[\sd(\sigma)]\rightarrow k[\sd(\sigma)])$ as graded $A$-modules.
  The grading of $\Theta A\otimes_A k[\sd(\sigma)]$ is given by	$deg(f\otimes_A g)=deg_A (f)+deg_A (g)$, where $deg_A$ refers to the grading induced by $A$.

As mentioned before, for generic $\Theta$, $\widetilde{\Theta}:=\{\theta_1,\ldots,\theta_{d-1}\}$ is a l.s.o.p. for $k[\sd(\sigma)]$. Thus the kernel of the map
$$(\Theta A\otimes_A k[\sd(\sigma)])_i\rightarrow (k[\sd(\sigma)])_i; \hspace{5pt} b\otimes f\mapsto bf$$
is zero iff the kernel of the map
$$((\theta_d) \otimes_A (k[\sd(\sigma)]/\widetilde{\Theta}))_i\rightarrow (k[\sd(\sigma)]/\widetilde{\Theta})_i;\hspace{5pt} \theta_d\otimes f\mapsto \theta_d f$$
is zero, which is the case iff the kernel of the multiplication map
$$\theta_d:\hspace{5pt} (k[\sd(\sigma)]/\widetilde{\Theta})_{i-1}\rightarrow (k[\sd(\sigma)]/\widetilde{\Theta})_i;\hspace{5pt} f\mapsto \theta_d f$$
is zero. (We have a shift $(-1)$ in the grading since the last map $\theta_d$ increases the degree by $+1$).

By construction, $\sigma$ is a pure subcomplex of the boundary of a $(d-1)$-simplex and thus is shellable. Since $\dim(\sigma)=d-2$ the induction hypothesis applies to $\sd(\sigma)$. Thus, multiplication
$$\theta_d^{d-2i-2}:\hspace{5pt} (k[\sd(\sigma)]/\widetilde{\Theta})_i\rightarrow (k[\sd(\sigma)]/\widetilde{\Theta})_{d-i-2}$$
is an injection for $0\leq i\leq \lfloor \frac{d-3}{2}\rfloor$ for a generic degree one element $\theta_d$. In particular, multiplication
 $$\theta_d:\hspace{5pt} (k[\sd(\sigma)]/\widetilde{\Theta})_i\rightarrow (k[\sd(\sigma)]/\widetilde{\Theta})_{i+1}$$
 is injective as well. Thus, $\Tor_1(k[\sd(\sigma)],A/\Theta)_i=0$ for $1\leq i\leq \lfloor \frac{d-3}{2}\rfloor +1 = \lfloor \frac{d-1}{2}\rfloor$. In particular, $\Tor_1(k[\sd(\sigma)],A/\Theta)_i=0$ for $1\leq i\leq\lfloor \frac{d-2}{2}\rfloor$. Note that $(\Theta A\otimes_A k[\sd(\sigma)])_0=0$, hence $\Tor_1(k[\sd(\sigma)],A/\Theta)_0=0$. To summarize,
 $\Tor_1(k[\sd(\sigma)],A/\Theta)_i=0$ for $0\leq i\leq \lfloor\frac{d-2}{2}\rfloor$, which completes the proof.
\end{proof}

\section{Combinatorial Consequences} \label{CombCons}
We are now going to discuss some combinatorial consequences of Theorem \ref{mainresult}.

For a $(d-1)$-dimensional simplicial complex $\Delta$ let its $g$-vector be $g^{\Delta}:=(g_0^{\Delta},g_1^{\Delta},\ldots,g_{\lfloor\frac{d}{2}\rfloor}^{\Delta})$, where $g_0^{\Delta}=1$ and $g_i^{\Delta}=h_i^{\Delta}-h_{i-1}^{\Delta}$ for $1\leq i\leq \lfloor\frac{d}{2}\rfloor$.
A sequence $(a_0,\ldots,a_t)$ is called an $M$-sequence if it is the Hilbert function of a standard graded Artinian $k$-algebra. Macaulay \cite{Macaulay} gave a characterization of such sequences by means of numerical conditions among their elements (see e.g. \cite{StanleyGreenBook}).

Recall that shellable complexes are CM (e.g. \cite{Bruns-Herzog}). While the converse is not true, Stanley showed that these two families of complexes have the same set of $h$-vectors:
\begin{theorem}[Theorem 3.3 \cite{StanleyGreenBook}]\label{ShellableCM}
Let $s=(h_0,\ldots,h_d)$ be a sequence of integers. The following conditions are equivalent:
\begin{itemize}
\item[(i)] $s$ is the $h$-vector of a shellable simplicial complex.
\item[(ii)] $s$ is the $h$-vector of a Cohen-Macaulay simplicial complex.
\item[(iii)] $s$ is an $M$-sequence.
\end{itemize}
\end{theorem}
\noindent We are now going to prove some $f$-vector corollaries of Theorem \ref{mainresult}, using Theorem \ref{ShellableCM} and Theorem \ref{Brenti, Welker}.

\begin{proof}[Proof of Corollary \ref{cor:CM,M-seq}]
For any simplicial complex $\Gamma$, $h^{\sd(\Gamma)}$ is a function of $h^{\Gamma}$. (For an explicit formula, see Theorem \ref{Brenti, Welker} below, obtained in \cite{Brenti-Welker}.)
Hence, together with Theorem \ref{ShellableCM} we can assume that $\Delta$ is shellable. Let $\dim\Delta=d-1$.

By Theorem \ref{mainresult}, for a generic l.s.o.p. $\Theta$ and a generic degree one element $\omega$, multiplication
$$\omega^{d-1-2i}:\hspace{5pt} (k[\sd(\Delta)]/\Theta)_i\rightarrow(k[\sd(\Delta)]/\Theta)_{d-1-i}$$
is an injection for $0\leq i\leq \lfloor\frac{d-2}{2}\rfloor$, hence $$\omega:\hspace{5pt} (k[\sd(\Delta)]/\Theta)_i\rightarrow(k[\sd(\Delta)]/\Theta)_{i+1}$$
is an injection as well. (This conclusion is vacuous for $d\leq 1$.)
Therefore, as Cohen-Macaulayness implies $h_i^{\sd(\Delta)}=\dim_k(k[\sd(\Delta)]/\Theta)_i$, we get that $g_i^{\sd(\Delta)}=\dim_k (k[\sd(\Delta)]/(\Theta, \omega))_i$ for $0\leq i\leq \lfloor\frac{d}{2}\rfloor$. Hence $g^{\sd(\Delta)}$ is an $M$-sequence.
\end{proof}

\begin{corollary}\label{cor:CM,ineq}
Let $\Delta$ be a $(d-1)$-dimensional Cohen-Macaulay simplicial complex. Then $h_{d-i-1}^{\sd(\Delta)}\geq h_i^{\sd(\Delta)}$ for any $0\leq i\leq \lfloor \frac{d-2}{2}\rfloor$.
\end{corollary}

\begin{proof}
Again, by Theorems \ref{Brenti, Welker} and \ref{ShellableCM} we can assume that $\Delta$ is shellable.
By Theorem \ref{mainresult}, for a generic l.s.o.p. $\Theta$ of $k[\sd(\Delta)]$ multiplication
$$\omega^{d-1-2i}:\hspace{5pt} (k[\sd(\Delta)]/\Theta)_i\rightarrow(k[\sd(\Delta)]/\Theta)_{d-1-i}$$
is an injection for $1\leq i\leq \lfloor\frac{d-2}{2}\rfloor$ and a generic degree one element $\omega$.
Since $h_i^{\sd(\Delta)}=\dim_k (k[\sd(\Delta)]/\Theta)_i$ this implies $h_i^{\sd(\Delta)}\leq h_{d-1-i}^{\sd(\Delta)}$.
\end{proof}

Next we verify the almost strong Lefschetz property for polytopal complexes. The proof essentially follows the same steps as the one of Theorem \ref{mainresult}; we will indicate the differences.

A \emph{polytopal complex} is a finite, non-empty collection $\mathcal{C}$ of polytopes (called the faces of $\mathcal{C}$) in some $\mathbb{R}^t$ that contains all the faces of its polytopes, and such that the intersection of two of its polytopes is a face of each of them. Notions like facets, dimension, pureness and barycentric subdivision are defined as usual.

Shellability extends to polytopal complexes as follows (see e.g. \cite{Ziegler} for more details).

\begin{definition}
Let $\mathcal{C}$ be a pure $(d-1)$-dimensional polytopal complex. A \emph{shelling} of $\mathcal{C}$ is a linear ordering $F_1,F_2,\ldots,F_m$ of the facets of $\mathcal{C}$ such that either $\mathcal{C}$ is $0$-dimensional, or it satisfies the following conditions:
\begin{itemize}
	\item[(i)] The boundary complex $\mathcal{C}(\partial F_1)$ of the first facet $F_1$ has a shelling.
	\item[(ii)] For $1<j\leq m$ the intersection of the facet $F_j$ with the previous facets is non-empty and is the beginning segment of a shelling of the $(d-2)$-dimensional boundary complex of $F_j$, that is,
	$$F_j\cap(\bigcup_{i=1}^{j-1}F_i)=G_1\cup G_2\cup\ldots\cup G_r$$
	for some shelling $G_1,G_2,\ldots,G_r,\ldots,G_t$ of $\mathcal{C}(\partial F_j)$, and $1\leq r \leq t$.\\
\end{itemize}
A polytopal complex is called \emph{shellable} if it is pure and has a shelling.
\end{definition}

\begin{theorem}\label{cor:ShellablePolytopalComplex}
Let $\Delta$ be a shellable $(d-1)$-dimensional polytopal complex. Then $\sd(\Delta)$ is almost strong Lefschetz over $\mathbb{R}$. In particular, $g^{\sd(\Delta)}$ is an $M$-sequence.
\end{theorem}

\begin{proof}
We give a sketch of the proof, indicating the needed modifications w.r.t. the proof of Theorem \ref{mainresult}.

We use induction on the dimension and the number of facets $f_{d-1}^{\Delta}$ of $\Delta$. For $f_{d-1}^{\Delta}=1$, note that the barycentric subdivision of a polytope is combinatorially isomorphic to a simplicial polytope (see \cite{EwaldShephard}).


Theorem \ref{thm:HLforPolytopes} implies that $\sd(\partial P)$ is $(d-1)$-Lefschetz over $\mathbb{R}$. By Lemma \ref{SL:cone} the same holds for $\cone(\sd(\partial P))=\sd(P)$. Together with the $\dim\Delta=0$ case, this provides the base of the induction.

The induction step works as in the proof of Theorem \ref{mainresult}.
\end{proof}
Note that in the above proof we really need the classical $g$-theorem, whereas in the proof of Theorem \ref{mainresult} it was not required.

\section{New Inequalities for the refined Eulerian statistics on permutations} \label{permutations}
In \cite{Brenti-Welker} Brenti and Welker give a precise description of the $h$-vector of the barycentric subdivision of a simplicial complex in terms of the $h$-vector of the original complex.
The coefficients that occur in this representation are a refinement of the Eulerian statistics on permutations.

Let $S_d$ denotes the symmetric group on $[d]$, and let $\sigma\in S_d$.
We write $D(\sigma):=\{~i\in[d-1]~|~\sigma(i)>\sigma(i+1)\}$ for the descent set of $\sigma$ and $\des(\sigma):=\#D(\sigma)$ counts the number of descents of $\sigma$. For $0\leq i\leq d-1$ and $1\leq j\leq d$ we set $A(d,i,j):=\#\{~\sigma\in S_d~|~des(\sigma)=i,\;\sigma(1)=j\}$. Brenti and Welker showed that these numbers satisfy the following symmetry:
\begin{lemma} \cite[Lemma 2.5]{Brenti-Welker} \label{sym}
$$A(d,i,j)=A(d,d-1-i,d+1-j)$$
for $d\geq 1$, $1\leq j\leq d$ and $0\leq i\leq d-1$.
\end{lemma}
The following theorem establishes the relation between the $h$-vector of a simplicial complex and the $h$-vector of its barycentric subdivision. As stated in \cite{Brenti-Welker} the result actually holds in the generality of Boolean cell complexes.

\begin{theorem}\cite[Theorem 2.2]{Brenti-Welker}\label{Brenti, Welker}
Let $\Delta$ be a $(d-1)$-dimensional simplicial complex and let $\sd(\Delta)$ be its barycentric subdivision. Then
$$h_j^{\sd(\Delta)}=\sum_{r=0}^d A(d+1,j,r+1)h_r^{\Delta}$$
for $0\leq j\leq d$.
\end{theorem}
In order to prove some new inequalities for the $A(d,i,j)$'s we will need the following characterization of the $h$-vector of a shellable simplicial complex due to McMullen and Walkup.

\begin{proposition}\cite[Corollary 5.1.14]{Bruns-Herzog} \label{McMullen, Walkup}
Let $\Delta$ be a shellable $(d-1)$-dimensional simplicial complex with shelling $F_1,\ldots,F_m$. For $2\leq j\leq m$, let $r_j$ be the number of facets of $\langle F_j\rangle\cap\langle F_1,\ldots,F_{j-1}\rangle$ and set $r_1=0$. Then
$h_i^{\Delta}=\#\{~j~|~r_j=i\}$ for $i=0,\ldots,d$. In particular, the numbers $h_j^{\Delta}$ do not depend on the particular shelling.
\end{proposition}
It is easily seen that $r_j=\#\res(F_j)$. We will use this fact in the proof of the following corollary.

\begin{corollary}\label{inequalities}
\begin{itemize}
\item[(i)] $A(d+1,j,r)\leq A(d+1,d-1-j,r)$ for $d\geq 0$, $1\leq r\leq d+1$ and $0\leq j\leq \lfloor\frac{d-2}{2}\rfloor$.
	\item[(ii)] $$A(d+1,0,r+1)\leq A(d+1,1,r+1)\leq\ldots\leq A(d+1,\lfloor\frac{d}{2}\rfloor,r+1)$$
	and
	$$A(d+1,d,r+1)\leq A(d+1,d-1,r+1)\leq\ldots\leq A(d+1,\lceil\frac{d}{2}\rceil,r+1)$$
	 for $d\geq 1$ and $1\leq r\leq d$.(For $d$ odd, $A(d+1,\lfloor \frac{d}{2}\rfloor,r+1)$ may be larger or smaller then $A(d+1,\lceil \frac{d}{2}\rceil,r+1)$.)
\end{itemize}
\end{corollary}

\begin{proof}
Let $\Delta$ be a shellable $(d-1)$-dimensional simplicial complex. Let $F_1,\ldots, F_m$ be a shelling of $\Delta$ with $m\geq 2$ and set $\widetilde{\Delta}:=\langle F_1,\ldots,F_{m-1}\rangle$. Since $\sd(\widetilde{\Delta})$ is a subcomplex of $\sd(\Delta)$ we get the following short exact sequence of $A$-modules for $A=k[x_1,\ldots,x_{f_0^{\sd(\Delta)}}]$:
$$0\rightarrow I\rightarrow k[\sd(\Delta)]\rightarrow k[\sd(\widetilde{\Delta})]\rightarrow 0,$$
where $I$ denotes the kernel of the projection on the right-hand side. Let $\Theta$ be a maximal l.s.o.p. for both $k[\sd(\Delta)]$ and $k[\sd(\widetilde{\Delta})]$. As $\widetilde{\Delta}$ is shellable it is CM and therefore $\sd(\widetilde{\Delta})$ is CM as well. Hence dividing out by $\Theta$ yields the following exact sequence of $A$-modules:
\begin{equation} \label{exactSequence}
0\rightarrow I/(I\cap\Theta)\rightarrow k[\sd(\Delta)]/\Theta\rightarrow k[\sd(\widetilde{\Delta})]/\Theta\rightarrow 0.
\end{equation}
Consider the following commutative diagram
$$\begin{array}{ccccccccc}\label{commDiagram}
   0 & \to & I/(I \cap \Theta)_i &\to& (k[\sd(\Delta)]/\Theta)_i& &\\
    & & & & & & &\\
   & & \downarrow \omega^{d-1-2i}& &\ \ \downarrow \omega^{d-1-2i}& &\\
     & & & & & & &\\
      0 & \to & I/(I \cap \Theta)_{d-1-i} &\to& (k[\sd(\Delta)]/\Theta)_{d-1-i} & & \\
   \end{array}$$
where $\omega$ is in $A_1$.
By Theorem \ref{mainresult} multiplication
$$\omega^{d-2i-1}:\hspace{5pt}(k[\sd(\Delta)]/\Theta)_i\rightarrow (k[\sd(\Delta)])_{d-1-i}$$
is an injection for $0\leq i\leq \lfloor\frac{d-2}{2}\rfloor$ and generic $\omega$. It hence follows that also multiplication
\begin{equation}\label{multip}
\omega^{d-1-2i}: \hspace{5pt}(I/(I\cap\Theta))_i\rightarrow (I/(I\cap\Theta))_{d-1-i}
\end{equation}
is an injection for $0\leq i\leq \lfloor\frac{d-2}{2}\rfloor$.

Furthermore, we deduce from the sequence (\ref{exactSequence}) that $\dim_k(I/(I\cap\Theta))_t=h_t^{\sd(\Delta)}-h_t^{\sd(\widetilde{\Delta})}$ for $0\leq t\leq d$.

In order to compute this difference we determine the change in the $h$-vector of $\widetilde{\Delta}$ when adding the last facet $F_m$ of the shelling. Let $r_m:=\#\res(F_m)$. Proposition \ref{McMullen, Walkup} implies $h_{r_m}^{\Delta}=h_{r_m}^{\widetilde{\Delta}}+1$ and $h_i^{\Delta}=h_i^{\widetilde{\Delta}}$ for $i\neq r_m$. Using Theorem \ref{Brenti, Welker} we deduce:
\begin{eqnarray*}
h_{i}^{\sd(\Delta)}&=&\sum_{r=0}^d A(d+1,i,r+1)h_r^{\Delta}\\
&=&\sum_{r=0}^d A(d+1,i,r+1)h_r^{\widetilde{\Delta}}+A(d+1,i,r_m+1)\\
&=&h_i^{\sd(\widetilde{\Delta})}+A(d+1,i,r_m+1).
\end{eqnarray*}
Thus $\dim_k(I/(I\cap\Theta))_i=A(d+1,i,r_m+1)$ for $0\leq i\leq \lfloor \frac{d-2}{2}\rfloor$. {From} (\ref{multip}) it follows that
$A(d+1,i,r_m+1)\leq A(d+1,d-1-i,r_m+1)$.

Take $\Delta$ to be the boundary of the $d$-simplex. Since in this case $h_i^{\Delta}\geq 1$ for $0\leq i\leq d$, i.e. restriction faces of all possible sizes occur in a shelling of $\Delta$, it follows that
$A(d+1,i,r)\leq A(d+1,d-1-i,r)$ for every $1\leq r\leq d+1$ and $0\leq i\leq \lfloor\frac{d-2}{2}\rfloor$. This shows (i).

To show (ii) we use that the injections in (\ref{multip}) induce injections
$$\omega: \hspace{5pt}(I/(I\cap\Theta))_i\rightarrow (I/(I\cap\Theta))_{i+1}$$
for $0\leq i\leq \lfloor\frac{d-2}{2}\rfloor$.
Thus, $A(d+1,i,r_m+1)\leq A(d+1,i+1,r_m+1)$. The same reasonning as in (i) shows that $A(d+1,i,r)\leq A(d+1,i+1,r+1)$ for $0\leq i\leq \lfloor\frac{d-2}{2}\rfloor$ and $1\leq r\leq d$.
The second part of (ii) follows from the first one using Lemma \ref{sym}.
\end{proof}

\begin{example}
$A(6,2,3)=60>48=A(6,3,3)$ while $A(6,2,4)=48<60=A(6,3,4)$.
This shows that for $d$ odd $A(d+1,\lfloor \frac{d}{2}\rfloor,r+1)$ may be larger or smaller then $A(d+1,\lceil \frac{d}{2}\rceil,r+1)$.
\end{example}

Recall that a sequence of integers $s=(s_0,\ldots,s_d)$ is called unimodal if there is a $0\leq j\leq d$ such that $s_0\leq\ldots \leq s_j\geq \ldots\geq s_d$. We call $s_j$ a \emph{peak} of this sequence and say that it is at \emph{position $j$} (note that $j$ may not be unique).

\begin{remark}
{From} \cite{Brenti-Welker} it can already be deduced that the sequence $(A(d+1,0,r+1),\ldots,A(d+1,d,r+1))$ is unimodal. Applying the linear transformation of Theorem \ref{Brenti, Welker} to the $(r+1)$st unit vector yields the sequence $(A(d+1,0,r+1),\ldots,A(d+1,d,r+1))$. It then follows from \cite[Theorem 3.1, Remark 3.3]{Brenti-Welker} that the generating polynomial of this sequence is real-rooted.
Since $A(d,i,r+1)\geq 1$ for $i\geq 1$ the sequence $(A(d+1,0,r+1),\ldots,A(d+1,d,r))$ has no internal zeros. Together with the real-rootedness this implies that $(A(d+1,0,r+1),\ldots,A(d+1,d,r))$ is unimodal. However, this argument tells nothing about the position of the peak.
\end{remark}

Recall that a regular CW-complex $\Delta$ is called a \emph{Boolean cell complex} if for each $A\in\Delta$ the lower interval $[\emptyset,A]:=\{~B\in\Delta~|~\emptyset\leq_{\Delta}B\leq_{\Delta}A\}$ is a Boolean lattice, where $A\leq_{\Delta}A'$ if $A$ is contained in the closure of $A'$ for $A,A'\in\Delta$.
In \cite{Brenti-Welker} it was shown that the $h$-vector of the barycentric subdivision of a Boolean cell complex with
non-negative entries is unimodal. What remains open is the location of its peak. Using Corollary \ref{inequalities} we can solve this problem:

\begin{corollary}\label{cor:Peak}
Let $\Delta$ be a $(d-1)$-dimensional Boolean cell complex with $h_i^{\Delta}\geq 0$ for $0\leq i\leq d$. Then the peak of $h^{\sd(\Delta)}$ is at position $\frac{d}{2}$ if $d$ is even and at position $\frac{d-1}{2}$ or $\frac{d+1}{2}$ if $d$ is odd. In particular, this assertion holds for CM complexes.
\end{corollary}

\begin{proof}
Since $h_i^{\Delta}\geq 0$ for $0\leq i\leq d$, by Theorem \ref{Brenti, Welker} and Corollary \ref{inequalities} (ii) we deduce
\begin{eqnarray*}
h_j^{\sd(\Delta)}&=&\sum_{r=0}^d A(d+1,j,r+1)h_r^{\Delta}\\
&\stackrel{Corollary\; \ref{inequalities} (ii)}{\leq}&\sum_{r=0}^d A(d+1,j+1,r+1)h_r^{\Delta}=h_{j+1}^{\sd(\Delta)}
\end{eqnarray*}
for $0\leq j\leq \lfloor\frac{d-2}{2}\rfloor$. Thus $h_0^{\sd(\Delta)}\leq h_1^{\sd(\Delta)}\leq\ldots\leq h_{\lfloor\frac{d}{2}\rfloor}^{\sd(\Delta)}$.

Similarly one shows $h_{\lceil\frac{d}{2}\rceil}^{\sd(\Delta)}\geq h_{\lceil\frac{d}{2}\rceil+1}^{\sd(\Delta)}\geq\ldots\geq h_d^{\sd(\Delta)}$, when applying Corollary \ref{inequalities} (ii) for $j\geq \lceil \frac{d}{2}\rceil$.

If $d$ is even $\lfloor\frac{d}{2}\rfloor=\lceil\frac{d}{2}\rceil=\frac{d}{2}$ and the peak of $h^{\sd(\Delta)}$ is at position $\frac{d}{2}$.
\end{proof}

\begin{example}
If $d$ is odd, depending on whether $h_{\lfloor\frac{d}{2}\rfloor}^{\sd(\Delta)}\leq h_{\lceil\frac{d}{2}\rceil}^{\sd(\Delta)}$ or vice versa the peak of $h^{\sd(\Delta)}$ is at position $\frac{d-1}{2}$ or $\frac{d+1}{2}$. For example,
for $d=3$ let $\Delta$ be the $2$-skeleton of the $4$-simplex. Then $h^{\Delta}=(1,2,3,4)$ and $h^{\sd(\Delta)}=(1,22,33,4)$, i.e. the peak is at position $\frac{3+1}{2}=2$.\\
If $\Delta$ consists of $2$ triangles intersecting along one edge, i.e. $\Delta:=\langle \{1,2,3\},\{2,3,4\} \rangle$, then
$h^{\Delta}=(1,1,0,0)$ and $h^{\sd(\Delta)}=(1,8,3,0)$. In this case the $h$-vector peaks at position $\frac{3-1}{2}=1$.
\end{example}

Using Corollary \ref{inequalities} we establish also the following inequalities; compactly summarized later in Corollary \ref{cor:vectorA(d)}.

\begin{corollary} \label{inequalities2}
\begin{itemize}
\item[(i)] $A(d+1,j,1)\leq A(d+1,j,2)\leq \ldots\leq A(d+1,j,d+1)$ for $\lceil\frac{d+1}{2}\rceil=\lfloor\frac{d+2}{2}\rfloor\leq j\leq d$.
\item[(ii)] $A(d+1,j,1)\geq A(d+1,j,2)\geq\ldots\geq A(d+1,j,d+1)$ for $0\leq j\leq \lfloor\frac{d-1}{2}\rfloor$.
\item[(iii)] $A(d+1,\frac{d}{2},1)\leq A(d+1,\frac{d}{2},2)\leq \ldots\leq A(d+1,\frac{d}{2},\frac{d}{2}+1)\geq A(d+1,\frac{d}{2},\frac{d}{2}+2)\geq \ldots\geq A(d+1,\frac{d}{2},d+1)$ if $d$ is even.
\item[(iv)] $A(d+1,j,1)=A(d+1,j+1,d+1)$ for $0\leq j\leq d-1$.
\end{itemize}
\end{corollary}

\begin{proof}
To prove (i) we need to show that $A(d+1,j,r)\leq A(d+1,j,r+1)$ for $1\leq r\leq d$ and $\lfloor\frac{d+2}{2}\rfloor\leq j\leq d$. For $j=d$ this follows from $\{~\sigma\in S_{d+1}~|~\des(\sigma)=d\}=\{~(d+1)d\ldots 21~\}$. Let $C_{j,r}^d:=\{~\sigma\in S_{d+1}~|~\des(\sigma)=j,\sigma(1)=r\}$. Consider the following map:

$$
\begin{array}{ccc}
\phi_{j,r}^d:\{\sigma\in C_{j,r}^d~|~\sigma(2)\neq r+1\}&\rightarrow& \{~\sigma\in C_{j,r+1}^d~|~\sigma(2)\neq r\}\\
\sigma&\mapsto& (r,r+1)\sigma.
\end{array}
$$
For $\sigma\in C_{j,r}^d$, if $\des(\sigma)=j$ and $\sigma(2)\neq r+1$ then $\sigma$ and $(r,r+1)\sigma$ have the same descent set, hence $\des((r,r+1)\sigma)=j$ as well.

As $((r,r+1)\sigma)(1)=r+1$ the function $\phi_{j,r}^d$ is well-defined. Since $(r,r+1)^2=id$ it follows that $\phi_{j,r}^d$ is invertible and therefore $\#\{~\sigma\in C_{j,r}^d~|~\sigma(2)\neq r+1\}=\#\{~\sigma\in C_{j,r+1}^d~|~\sigma(2)\neq r\}$.\\
If $\sigma\in C_{j,r}^d$ and $\sigma(2)=r+1$, then all of the $j$ descents must occur at position at least $2$.

The sequence $\tilde{\sigma}=(r+1)\sigma(3)\ldots\sigma(d+1)$ can be identified with a permutation $\tau$ in $S_d$ with $\tau(1)=r$ and vice versa via the order preserving map $[d+1]\setminus \{r\}\rightarrow [d]$, hence the descent set is preserved under this identification.
Therefore
 $\#\{~\sigma\in C_{j,r}^d~|~\sigma(2)= r+1\}=\#\{~\sigma\in C_{j,r}^{d-1}\}=A(d,j,r)$. On the other hand, if $\sigma\in C_{j,r+1}^d$ and $\sigma(2)=r$ then $\sigma$ has exactly $j-1$ descents at positions $\{2,\ldots,d\}$. A similar argumentation as before then implies  $\#\{~\sigma\in C_{j,r+1}^d~|~\sigma(2)= r\}=\#\{~\sigma\in C_{j-1,r}^{d-1}\}=A(d,j-1,r)$.

By Corollary \ref{inequalities}(ii) it holds that $A(d,j,r)\leq A(d,j-1,r)$ for $d-2\geq j-1\geq \lceil\frac{d-1}{2}\rceil$, i.e. $d-1\geq j\geq \lceil\frac{d+1}{2}\rceil=\lfloor\frac{d+2}{2}\rfloor$.
Combining the above, we obtain $A(d+1,j,r)\leq A(d+1,j,r+1)$ for $1\leq r\leq d$ and $\lfloor\frac{d+2}{2}\rfloor\leq j\leq d-1$, and (i) follows.

(ii) follows directly from (i) and Lemma \ref{sym}.

For the proof of (iii) we only show $A(d+1,\frac{d}{2},1)\leq A(d+1,\frac{d}{2},2)\leq \ldots\leq A(d+1,\frac{d}{2},\frac{d}{2}+1)$. The other inequalities in (iii) follow directly from this part by Lemma \ref{sym}.
The proof of (i) shows that $\#\{~\sigma\in C_{\frac{d}{2},r}^d~|~\sigma(2)\neq r+1\}=\#\{~\sigma\in C_{\frac{d}{2},r+1}^d~|~\sigma(2)\neq r\}$.
As in the proof of (i), it remains to prove that $A(d,\frac{d}{2},r)\leq A(d,\frac{d}{2}-1,r)$ for $1\leq r\leq \frac{d}{2}$. By Lemma \ref{sym} it holds that $A(d,\frac{d}{2},r)=A(d,\frac{d}{2}-1,d+1-r)$. For $1\leq r\leq \frac{d}{2}$ we have $r\leq d+1-r$ and (ii) then implies $A(d,\frac{d}{2}-1,r)\geq A(d,\frac{d}{2}-1,d+1-r)$ which finishes the proof of (iii).

To show (iv) note that by Lemma \ref{sym} $A(d+1,j,1)=A(d+1,d-j,d+1)$. If $\sigma=(d+1)\sigma(2)\ldots\sigma(d+1)\in C_{d-j,d+1}^d$, then the 'reverse' permutation $\widetilde{\sigma}:=(d+1)\sigma(d+1)\ldots\sigma(2)$ has a descent at position $1$ and whenever there is an ascent in $\sigma(2)\ldots\sigma(d+1)$. Since $\sigma$ has $d-j-1$ descents at positions $\{2,\ldots,d\}$ this implies $\des(\widetilde{\sigma})=1+(d-1)-(d-j-1)=j+1$, i.e. $\widetilde{\sigma}\in C_{j+1,d+1}^{d}$. We recover $\sigma$ by repeating this construction and hence $A(d+1,j,1)=A(d+1,j+1,d+1)$.
\end{proof}

Let $\mathcal{A}:=(A(d,i,j))_{i,j}$ be the matrix with entries $A(d,i,j)$ for fixed $d$. For pairs $(i,j), (i',j')$ we set $(i,j)<(i',j')$ if either $i< i'$ or $i=i'$ and $j>j'$. This defines a total order on the set of pairs $(i,j)$.
Using this ordering for the indices of the entries of the matrix we can write the matrix $\mathcal{A}$ as a vector $A(d)$.

{From} Corollary \ref{inequalities2} and Lemma \ref{sym} we immediately get the following.

\begin{corollary}\label{cor:vectorA(d)}
The sequence $A(d)$ is unimodal and symmetric for $d\geq 1$. In particular, the peak of $A(d)$ lies in the middle. $\square$
\end{corollary}

The numerical results in Corollaries \ref{cor:CM,M-seq} and \ref{cor:Peak}
suggest that the barycentric subdivision of a Cohen-Macaulay simplicial complex might be weak Lefschetz. Recall that a $(d-1)$-dimensional simplicial complex is called \emph{weak Lefschetz over $k$} if there exists a maximal l.s.o.p. $\Theta$ for $k[\Delta]$ and a degree one element $\omega\in (k[\Delta]/\Theta)_1$ such that the multiplication maps
$$\omega:\hspace{5pt}(k[\Delta]/\Theta)_i\;\longrightarrow\;(k[\Delta]/\Theta)_{i+1}$$
have full rank for every $i$. In particular, in our case this means injections for $0\leq i< \frac{d}{2}$ and surjections for $\lceil\frac{d}{2}\rceil\leq i \leq\dim\Delta$.


\begin{problem}
Let $\Delta$ be a $(d-1)$-dimensional Cohen-Macaulay simplicial complex and $k$ be an infinite field. Then the barycentric subdivision of $\Delta$ is weak Lefschetz over $k$.
\end{problem}

\section*{acknowledgment}
We are grateful to Volkmar Welker for his helpful suggestions concerning earlier versions of this paper.
Further thanks go to Mike Stillman for pointing out a mistake in an earlier version.

\bibliography{biblio}
\bibliographystyle{plain}
\end{document}